\documentclass[Unicode]{arxiv-cedram-alco}
\usepackage{mathrsfs}

\equalenv{cor}{coro}
\equalenv{lem}{lemm}
\newtheorem*{nthm}{Greene's Theorem}

\newcommand{\set}[1]{\{#1\}}
\newcommand{\gr}[2]{\begin{center}
   \includegraphics[width=#1in]{#2}
\end{center}
}

\newcommand{\Z}{\mathbb Z}

\renewcommand{\L}{\mathscr{L}}
\newcommand{\A}{\mathscr{A}}
\newcommand{\B}{\mathscr{B}}
\newcommand{\I}{\mathscr{I}}

\newcommand{\al}[1]{\left\langle\left\langle #1\right\rangle\right\rangle}
\newcommand{\ring}{R\al{A^*}}

\title{An application of the Goulden-Jackson cluster theorem}

\author[\initial{I.} \middlename{M.} Gessel]{\firstname{Ira} \middlename{M.} \lastname{Gessel}}

\address{Department of Mathematics\\
   Brandeis University\\
   Waltham, MA 02453 (USA)}
 
\email{gessel@brandeis.edu}

\thanks{Supported by a grant from the Simons Foundation (\#427060, Ira Gessel).}

\keywords{Goulden-Jackson cluster theorem, forbidden subwords, M\"obius function}

\subjclass{05A05, 05A15}

\begin{document}

\begin{abstract}
Let $A$ be an alphabet and let 
$F$ be a set of words with letters in $A$. 
We show that the sum of all words with letters in $A$ with no consecutive subwords in $F$, 
as a formal power series in noncommuting variables,
is the reciprocal of a series with all coefficients 0, 1 or $-1$.
We also explain how this result is related to  the work of Dotsenko and Khoroshkin on a closely related problem and to a theorem of Curtis Greene on lattices with M\"obius function 0, 1, or $-1$. 
\end{abstract}

\maketitle

\section{Introduction}

Let $A$ be an alphabet, and let $A^*$ be the free monoid of words made up of letters in $A$, with the operation of concatenation. Let $R$ be a commutative ring with identity (in our application we may take it to be the polynomial ring $\mathbb{Z}[t]$), and let $\ring$ be the ring of formal sums of elements of $A^*$ with coefficients in $R$, multiplied in the obvious way. Then 
$\ring$ may be viewed as an algebra of formal power series in the noncommutative variables in $A$. We write 1 for the empty word, which is also the identity element of $\ring$, and we write $|w|$ for the length of the word $w$.

We call a word $u$ a \emph{subword} of a word $v$ if there exist words $p$ and $q$ such that $v=p u q$. Let $F$ be a set of nonempty words in $A^*$, and let $A_F$ be the set of words in $A^*$ in which no word in $F$ occurs as a subword. Since $1\in A_F$, the sum $\sum_{w\in A_F}w$ is invertible in $\ring$. Our main result describes the coefficients of the reciprocal of this sum.

\begin{thm}
\label{t-1}
For each word $v\in A^*$, let $M(v)$ be the coefficient of $v$ in 
\begin{equation*}
\biggl(\sum_{w\in A_F}w\biggr)^{-1}.
\end{equation*}
Then for all $v$, $M(v)$ is $0$, $1$, or $-1$.
\end{thm}

We derive Theorem \ref{t-1} from a noncommutative form of the Goulden-Jackson cluster theorem \cite{gj}, which gives a formula for counting words according to the number of subwords in $F$. Theorem \ref{t-1} is not an immediate consequence of the cluster theorem, since there is cancellation that must be accounted for. 

A similar result was proved by Dotsenko and Khoroshkin \cite{dk} using a different approach. A proof of Theorem \ref{t-1} was given by Iyudu and Vlassopoulos \cite[Corollary 4.1]{iv}. See also  Dotsenko, Vincent G\'elinas, and  Tamaroff \cite[section 1.1]{dgt}.
 
We also explain how Theorem \ref{t-1} is closely related to a result of Curtis Greene \cite{greene}, that lattices of unions of intervals, ordered by inclusion, have M\"obius functions that are $0$, $1$, or $-1$.

\section{The Goulden-Jackson Cluster Theorem}

Let $F$ be a set of  words in $A^*$ all of length at least 2. We call the elements of $F$ \emph{forbidden words}.
The cluster theorem allows us to count words in  $A^*$ by the number of forbidden subwords in terms of certain ``marked words" called \emph{clusters}.
Informally, a {cluster} is a word which is covered by an overlapping collection of marked forbidden subwords.
For example, if $A=\set{a}$ and $F=\set{aaa}$ then the following are both clusters on the word $a^6$:

\begin{equation}
\label{e-2clusters}
\vcenter{\hbox{\includegraphics[width=1in]{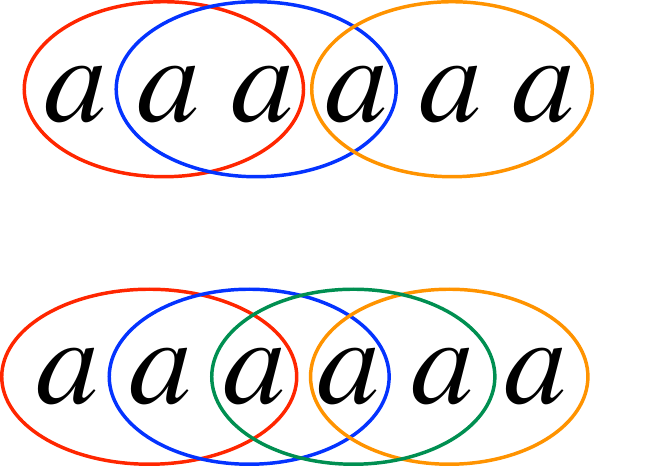}}}
\end{equation}

\noindent but \lower 5.5pt\hbox{\includegraphics[width=.9in]{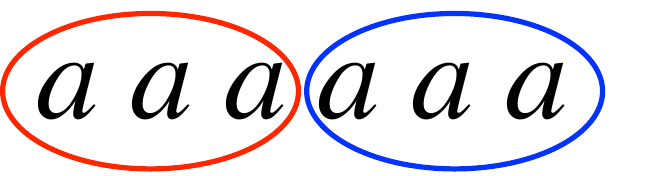}}
is not a cluster since the marked subwords don't overlap.

There are several possible formal definitions for marked words and clusters. The definition that we use is the most convenient for explaining the connection between our main result and Greene's theorem on M\"obius functions of unions of intervals; it also  allows the simple characterization of clusters given in Lemma \ref{l-cluster} below.

We use the notation $[m,n)$ for the interval of integers
$\{\,i\in \Z : m\le i < n\,\}$.
We define a \emph{marked word} (with respect to the alphabet $A$ and the set of forbidden words $F$) to be an ordered pair $(w, I)$ where $w=\alpha_1\alpha_2\cdots \alpha_n$ is a word in $A^*$ and $I$ is a set of intervals $[i,j)$, with $i < j \le |w|$, for which the subword $\alpha_{i}\alpha_{i+1}\cdots \alpha_{j}$ is in $F$. We call $w$ the \emph{underlying word} and $I$ the set of \emph{intervals}  of the marked word $(w,I)$. For consistency, we will say that the subword $\alpha_i\alpha_{i+1}\cdots \alpha_j$ of the word $\alpha_1\alpha_{2}\cdots \alpha_n$ occurs at the interval $[i,j)$. 
Note that for each letter $\alpha\in A$, $(\alpha, \varnothing)$ is a marked word.

The concatenation of two marked words $(u, I)$ and $(v, J)$ is defined by 
\begin{equation*}
(u,I) (v,J) = \bigl(uv, I \cup (J+|u|)\bigl),
\end{equation*}
where 
\begin{equation*}
J+|u| = \{\, [i+|u|, j+|u|) : [i,j) \in J\,\}.
\end{equation*}
Concatenation of marked words is easily seen to be associative and compatible with projection onto the underlying word. 

A marked word is a \emph{cluster} if its underlying word has length at least 2 and it cannot be expressed as a  concatenation of two nonempty marked words. 
We call a word $w$ a \emph{cluster word} if it is the underlying word of a cluster. 
The following characterization of clusters is clear:

\begin{lem}
\label{l-cluster}
A marked word $(u,I)$ is a cluster if and only if $|u|\ge2$ and
\begin{equation*}
\pushQED{\qed} 
\bigcup_{[i,j)\in I}[i,j) = [1,|u|).
\qedhere
\popQED
\end{equation*}
\end{lem}

For example, the two clusters shown in \eqref{e-2clusters} are formally $(a^6, \{[1,3), [2,4), [4,6)\})$ and $(a^6, \{[1,3), [2,4), [3,5),[4,6)\})$. The marked word after \eqref{e-2clusters} is $(a^6, \{[1,3), [4,6)\})$. It is not a cluster since it is the concatenation of $(a^3, \{[1,3)\})$ with itself, or alternatively, since 
$[1,3)\cup[4,6)=\{1,2,4,5\}\ne [1,6)$.

If we identify each letter $\alpha\in A$ with the marked word $(\alpha, \varnothing)$ then every marked word can be expressed uniquely as a concatenation of letters and clusters.

We define the \emph{cluster polynomial} $P_{F,w}(t)$ of a word $w$ by
\begin{equation*}
P_{F,w}(t)=\sum_{I}t^{|I|}
\end{equation*}
where the sum is over all $I$ for which  $(w,I)$ is a cluster.
For example, if $w=aaab$ and $F = \{aa,aab\}$ then the two clusters with underlying word $w$ are 
$[w,I)$ with $I=\{[1,2), [2,4)\}$ and $I=\{[1,2), [2,3), [2,4)\}$, 
so
$P_{F,w}(t) = t^2 + t^3$. 
If $w$ is not a cluster word then $P_{F,w}(t) = 0$.

We define the 
\emph{cluster generating function} for $F$ to be
\begin{equation*}
C_F(t) = \sum_w P_{F,w}(t) w
\end{equation*}
where the sum is over all words $w$ in $A^*$.

For a word $w\in A^*$, let $s_F(w)$ be the number of subwords of $w$ (counted with multiplicities) that are in $F$. For example, if $F=\{aa\}$ then $s_F(a^4) = 3$.

We  now state and prove our form of the Goulden-Jackson cluster theorem. (The original cluster theorem \cite{gj} has more general weights, but is commutative; the proofs are essentially the same.) For some applications of the cluster theorem, see Bassino, Cl\'ement, and Nicod\`eme \cite{bcn}, Noonan and Zeilberger \cite{nz}, Wang \cite{wang}, and Zhuang \cite{zhuang, zhuang-lift}.
\begin{thm}
\label{t-cluster}
Let $F\subseteq A^*$ be a set of words of length at least two.
Then
\begin{equation}
\label{e-cluster1}
\sum_{w\in A^*} t^{s_F(w)}w = 
 \biggl(1-\sum_{\alpha\in A}\alpha -C_F(t-1)\biggr)^{-1}.
\end{equation}
In particular, for $t=0$ we have
\begin{equation}
\label{e-t=0}
\sum_{w\in A_F} w = 
 \biggl(1-\sum_{\alpha\in A}\alpha -C_F(-1)\biggr)^{-1},
\end{equation}
where $A_F$ is the set of words in $A^*$ with no subwords in $F$.
\end{thm}

\begin{proof}
 Replacing $t$ by $t+1$ in \eqref{e-cluster1} gives the equivalent formula
\begin{equation}
\label{e-cluster2}
\sum_{w\in A^*} (1+t)^{s_F(w)}w = 
 \biggl(1-\sum_{\alpha\in A}\alpha -C_F(t)\biggr)^{-1},
\end{equation}
which is easy to prove directly: The coefficient of a word $w$ on the left side of \eqref{e-cluster2} counts marked words with underlying word $w$, where a marked word $(w,I)$ contributes $t^{|I|}$.  But since every marked word is a unique concatenation of clusters and letters, the right side of \eqref{e-cluster2} also counts marked words, with the same weights.
\end{proof}

As an example of Theorem \ref{t-cluster}, 
let $A=\set{a,b,c}$ and  $F=\set{abc, bcc}$. There are three clusters:
\gr{2}{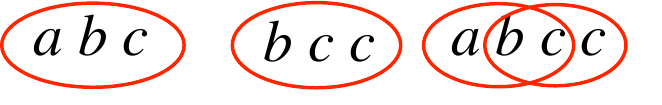}
Thus the cluster generating function $C_F(t)$ is $t\kern .5pt abc +t\kern .6pt bcc+t^2 abcc$.
So 
\begin{equation*}\sum_{w\in A^*}t^{s_F(w)}w  
=  \biggl(1-a-b-c-(t-1)abc -(t-1)bcc -(t-1)^2 abcc\biggr)^{-1}.
\end{equation*}

As another example, take $A=\{a\}$ and $F=\set{a^3}$. Although it's not hard to compute the cluster generating function directly, an indirect approach is even easier: For $n\ge 2$, there are $n-2$ occurrences of $a^3$ in $a^n$. So
\begin{equation*}\sum_{w\in A^*}t^{s_F(w)}w  
= 1+a+a^2+ta^3 +t^2a^4+\cdots = 1+a + a^2(1-ta)^{-1}.
\end{equation*}
Replacing $t$ with $1+t$, applying \eqref{e-cluster2},  and simplifying gives
\begin{equation}
\label{e-a^3}
C_F(t) = ta^3\bigl(1-t(a+a^2)\bigr)^{-1}.
\end{equation}
Setting $t=-1$ in \eqref{e-a^3} gives
\begin{align}
\label{e-a^30}
C_F(-1) &= -a^3(1+a+a^2)^{-1}=-a^3(1-a)(1-a^3)^{-1}\notag\\
 &=-a^3+a^4 -a^6+a^7 -\cdots,
\end{align}
a formula that we will derive in another way in section \ref{s-noforb}.
 
\section{A recurrence for the cluster polynomial}

The cluster polynomial $P_{F,w}(t)$ can be computed by a simple recurrence that will be needed in the proof of the main theorem. In Lemma \ref{l-recur} below we require that  $F$ be \emph{reduced}; that is, no word in $F$ is a subword of another word in $F$. This condition makes the recurrence simpler, and the case of reduced $F$ is sufficient for the proof of Theorem \ref{t-1}. We also require that $w$ be a cluster word;  if $w$ is not a cluster word then $P_{F,w}(t)=0$.

\begin{lem}
\label{l-recur}
Suppose that $F\subseteq A^*$ is a reduced set of forbidden words, and 
that $w=\alpha_1\alpha_2\cdots \alpha_n$ is a cluster word with respect to $F$. Then there exists a positive integer $m$, polynomials $p_1,p_2,\dots, p_m$ in $t$, and positive integers $r_2,r_3,\dots, r_m$, with $1\le r_k\le k-1$, such that
$p_1=t$, 
\begin{equation}
\label{e-prec}
p_k= t(p_{r_k}+p_{r_k+1}+\cdots+p_{k-1})\text{ for $2\le k\le m$},
\end{equation}
and $p_m = P_{F, w}(t)$.
\begin{proof}
Let the intervals of  the forbidden subwords of $w$ be $[i_1, j_1),\dots, [i_m,j_m)$, where $i_1<\cdots <i_m$ and thus \textup{(}since $F$ is reduced\textup{)} $j_1<\cdots <j_m$. 
Since $w$ is a cluster word, we  have $i_1=1$ and  $j_m=n$.
Let $w_k$ be the word $\alpha_1\alpha_2\cdots \alpha_{j_k}$ for $1\le k\le m$ and let $p_k = P_{F, w_k}(t)$. Then each $w_k$ is a cluster word and $w_m=w$. Since $w_1\in F$, we have $p_1=t$. Now suppose that $k>1$. Then $j_{k-1}\ge i_k$ since $w$ is a cluster word. So we may define $r_k$ to be the least integer such that  $j_{r_k}\ge i_k$, and we have $1\le r_k \le k-1$. In any cluster on $w_k$, the last interval must be $[i_k, j_k)$ and the next-to-last interval $[i_l, j_l)$ must satisfy $j_l\ge i_k$ and thus  $r_k\le l \le k-1$, and the contribution to $p_k$ from this value of $l$ is $tp_l$.  
Thus
\begin{equation*}
p_k= t(p_{r_k}+p_{r_k+1}+\cdots+p_{k-1}).\qedhere
\end{equation*}
\end{proof}
\end{lem}

For example, suppose that $F=\set{a^3}$ and $w = a^6$. Then the intervals of the forbidden words in $w$ are $[1,3)$, $[2,4)$, $[3,5)$, and $[4,6)$, so $m=4$ and we have $r_2=1$, $r_3=1$, and $r_4=2$. Then $p_1=t$ 
and the recurrence gives
\begin{align*}
p_2&=tp_1\\
p_3&=t(p_1+p_2)\\
p_4&=t(p_2+p_3)
\end{align*}
so $p_2=t^2$, $p_3=t^2+t^2$, and $p_4 = F_{F,w}(t)=2t^3+t^4$.

\section{Counting words without forbidden subwords}
\label{s-noforb}
The sum of all words in $A^*$ with no forbidden subwords is given by \eqref{e-t=0}. Thus we can prove Theorem \ref{t-1} by showing that $P_{F,w}(-1)$ is always 0, 1 or $-1$.

Let us first look at an important special case: Suppose that every forbidden word has length~$2$. Then the clusters are of the form $(\alpha_1 \alpha_2\cdots \alpha_n, I)$, where $\alpha_i \alpha_{i+1}\in F$ for $1\le i<n$ and $I=\{\,[i,i+1): 1\le i<n\,\}$, and the cluster polynomial for the word $\alpha_1 \alpha_2 \cdots \alpha_n$, when nonzero, is $t^{n-1}$.
In this case \eqref{e-t=0} may be written in the following symmetrical form:

\begin{cor}
\label{c-csv}
Let $F$ be a set of words of length $2$ in $A^*$. Let $\bar F$ be the complement of $F$ in $A^2$. Let $A_F$ be the set of words in which no word in $F$ occurs as a subword, and similarly for $A_{\bar F}$. \textup(So $A_F$ and $A_{\bar F}$ both contain the empty word and all words of length~$1$.\textup) Then
\begin{equation*}
\sum_{w\in A_F}w=\biggl(\sum_{w\in A_{\bar F}}(-1)^{|w|}w\biggr)^{-1}.
\end{equation*}
\end{cor}
Corollary \ref{c-csv} was first proved by Fr\"oberg \cite{froberg} (in a somewhat weaker form) and by Carlitz, Scoville, and Vaughan \cite{csv}, and was applied to various problems of permutation enumeration in Gessel \cite{gthesis}. Many related results can be found in Goulden and Jackson's book \cite[Chapter 4]{gjbook}. 

To prove our main result, Theorem \ref{t-main} below (a restatement of Theorem \ref{t-1}), we use Lemma \ref{l-recur} together with a lemma due to Curtis Greene \cite[Lemma 2.2]{greene}:
\begin{lem}
\label{l-greene}
Let $u_1, u_2,\dots, u_m $ be  integers satisfying  $u_1=-1$ and for $1<k\le m$, 
\begin{equation}\label {e-urec}
u_k = -(u_{r_k}+u_{r_k+1}\cdots+u_{k-1}),
\end{equation}
for some  integers $r_k$, where $1\le r_k \le k-1$. Then the nonzero entries of the sequence $u_1, u_2, u_3,\dots,u_m$ are $-1,1, -1, 1, -1, \dots$. 
\end{lem}

\begin{proof}  We are given that $u_1=-1$. Now suppose that $k>1$ and that the nonzero entries of $u_1,u_2,\dots, u_{k-1}$ are $-1,1,-1,1,\dots$. Then
$u_{r_k}+u_{r_k+1}\cdots+u_{k-1}$ must be 0, 1, or $-1$, and if the sum is nonzero then it must be equal to the last nonzero summand. Thus the nonzero entries of $u_1,u_2,\dots, u_{k}$ are $-1,1,-1,1,\dots$ and the result follows by induction.
\end{proof}

We can now prove our main result. 

\begin{thm}
\label{t-main}
The sum of all words in $A^*$ with no subwords in $F$ may be written
\begin{equation*}
\biggl(\sum_{w\in A^*} M(w) w\biggr)^{-1}
\end{equation*}
where for every word $w$, $M(w)$ is $0$, $1$, or $-1$.
\end{thm}
 
\begin{proof}
By equation \eqref{e-t=0} in Theorem \ref{t-cluster}, $M(w) = -P_{F,w}(-1)$, so it is sufficient to show that for every word $w$, $P_{F,w}(-1)$ is 0, 1, or $-1$. We will derive this from Lemma \ref{l-recur}.

We may assume without loss of generality that $F$ is reduced. To see this, note that if a word $u$ in $F$ is a subword of another word $v$ in $F$, then forbidding $u$ as a subword automatically forbids $v$, so we may remove $v$ from $F$ without changing $A_F$. Changing $F$ in this way may change $C_F(t)$ but it will not change $C_F(-1)$.

Then setting  $t=-1$  in Lemma \ref{l-recur} and applying Lemma \ref{l-greene} yields the theorem.
\end{proof}

Dotsenko and Khoroshkin \cite[Corollary 22]{dk} proved a result closely related to Theorem \ref{t-main}. Their interest was in finding exponential generating functions for  permutations avoiding consecutive patterns, so they did not consider words with repeated entries. However, their approach, based on earlier work of Anick \cite{anick}, can be used to prove Theorem \ref{t-main}, and gives an explicit, though recursive, description of the values of $M(w)$ in Theorem \ref{t-main}. A proof of Theorem \ref{t-main}, using algebraic techniques, was given by 
Iyudu and Vlassopoulos \cite[Corollary 4.1]{iv}. See also  Dotsenko, Vincent G\'elinas, and  Tamaroff \cite[section 1.1]{dgt}, which discusses ``Anick chains" and their connection with Tor groups of monomial algebras.

We can obtain a similar (though not obviously equivalent) description of $M(w)$ through a refinement of Lemma \ref{l-greene} that takes into account that $F$ is reduced:

\begin{lem}
\label{l-greene2}
With the assumptions of Lemma \textup{\ref{l-greene}}, suppose that in addition we have
$r_{k-1}\le r_k$ for $2< k \le m$.
Then in the sum on the right side of \eqref{e-urec} at most two terms are nonzero. Thus $u_k$ is nonzero if and only if exactly one of $u_{r_k},\dots, u_{k-1}$ is nonzero.
\end{lem}
\begin{proof}
We proceed by induction on $k$. The assertion is clearly true for $k=2$, since there is only one term in the sum. Now suppose that $k>2$ and that among
$u_{r_{k-1}},\dots ,u_{k-2}$, at most two are nonzero. 
Since $r_k\ge r_{k-1}$, at most two of $u_{r_k},\dots,u_{k-2}$ are nonzero. If fewer than two are nonzero, then at most two of $u_{r_k},\dots,u_{k-1}$ are nonzero, and if exactly two of  $u_{r_k},\dots,u_{k-2}$ are nonzero, then by \eqref{e-urec} (with $k-1$ for $k$), $u_{k-1}=0$, and the conclusion follows.
\end{proof}

Let us call a word $w$ in $A^*$  of length greater than 1 \emph{salient} (with respect to $F$) if $M(w)\ne 0$. Applying Lemma \ref{l-greene2} gives an explicit, though  recursive, characterization of salient words. First we note that by Theorem \ref{t-cluster}, every salient word must be a cluster word.

\begin{thm}
\label{e-sal}
Let $F$ be a reduced set of forbidden words, and let $w=\alpha_1\alpha_2\cdots \alpha_n$ be a cluster word.
If $w\in F$ then $w$ is salient with $M(w)=1$. Otherwise, suppose that the last forbidden subword in $w$ has interval $[j,n)$. Then $w$ is salient if and only if there is exactly one salient initial subword $w'=\alpha_1\alpha_2\cdots \alpha_m$ of $w$ with $j\le m<n$, and in this case $M(w) =-M(w')$.
\end{thm}

\begin{proof}
Since $F$ is reduced, the integers $r_k$ of Lemma \ref{l-recur} satisfy $r_{k-1}\le r_k$ for $2<k\le m$.
The theorem then follows by setting $t=-1$ in \eqref{e-prec} and applying Lemma \ref{l-greene2}.
\end{proof}

As an example of Theorem \ref{e-sal}, take $A=\{a\}$ and $F=\{a^3\}$. With the notation of Theorem \ref{e-sal}, let us call the  initial subwords $\alpha_1\alpha_2\cdots \alpha_m$ of $w$ with $j\le m<n$ the \emph{candidates} for $w$.  In this example, the cluster words are of the form $a^n$ with $n\ge3$, and the candidates for $a^n$, with $n\ge3$, are $a^{n-1}$ and $a^{n-2}$. 

The word $a^3$ is salient, since it is in $F$, and  $M(a^3) =1$. 
The candidates for $a^4$ are $a^3$ and $a^2$. Only $a^3$ is salient, so $a^4$ is salient with $M(a^4)=-1$. The candidates for $a^5$ are $a^4$ and $a^3$. Since both  are salient, $a^5$ is not. The candidates for $a^6$ are $a^5$ and $a^4$. Of these only $a^4$ is salient, so $a^6$ is salient with $M(a^6) =1$. In general, we can easily show by induction that for $n>3$, if $n\equiv 0\pmod 3$ then only candidate $a^{n-2}$ is salient, so $a^n$ is salient. If $n\equiv 1\pmod 3$ then only candidate $a^{n-1}$ is salient so $a^n$ is salient. However, if $n\equiv 2\pmod3$ then $a^{n-1}$ and $a^{n-2}$ are both salient, so $a^n$ is not salient. A similar analysis holds for $F=\{a^k\}$ for any $k\ge2$.

\section{Greene's theorem on M\"obius functions of lattices}

Theorem \ref{t-main} is closely related to a result of Curtis Greene \cite{greene},  which we will derive from it. We note that Greene's proof of  this result used Lemma \ref{l-greene}.

\begin{nthm}
Let $I_1, I_2,\dots, I_m$ be nonempty intervals in $\Z$. Let $\L$ be the lattice of unions of the $I_j$ \textup(including the empty set\textup) ordered by inclusion. Then for all $X\subseteq Y$ in $\L$ we have $\mu(X,Y)\in \{-1,0,1\}$, where $\mu$ is the M\"obius function of $\L$.
\end{nthm}

To prove Greene's theorem, we first recall a well-known formula for the M\"obius function of a lattice, a special case of Rota's cross-cut theorem \cite{rota}. We include a short proof for completeness.

Recall that an \emph{atom} of a lattice is an element that covers the minimal element~$\hat0$.

\begin{lem} 
\label{l-atoms}
Let $L$ be a finite lattice, and let $\A$ be the set of atoms of $L$. Then for any $x\in L$, the M\"obius function $\mu(\hat0, x)$ is given by 
\begin{equation}
\label{e-mob}
\mu_L (\hat0, x)  = \sum_\B (-1)^{|\B|},
\end{equation}
where the sum is over all subsets $\B\subseteq \A$ with join $x$.
\end{lem}
\begin{proof}
Let $f(x)$ be the sum on the right side of \eqref{e-mob}.
Then $\sum_{y\le x}f(y)=\sum_{\B\subseteq \A_x} (-1)^{|\B|}$, where $\A_x$ is the set of  atoms of $L$ less than or equal to $x$. Thus 
$\sum_{y\le \hat0} f(y) =f(\hat0)= 1$
and
if $x>\hat0$ then $\sum_{y\le x} f(y) = 0$.
Thus $f(x)$ satisfies the same recurrence that defines $\mu(\hat0,x)$, so $f(x) = \mu(\hat0,x)$.
\end{proof}

\begin{proof}[Proof of Greene's Theorem]
We prove the case in which $X=\hat0 = \varnothing$ and $Y=\hat 1 = I_1\cup I_2 \cup\cdots \cup I_m$; the general case follows easily from this case.

Let $\I= \{I_1, I_2, \dots, I_m\}$ be a set of nonempty intervals in $\Z$. Without loss of generality we may assume that $I_1\cup I_2\cup\cdots \cup I_m = [1,n)$, where $n\ge2$. Let $\L$ be the lattice of unions of the intervals in $\I$. Let $\I'$ be the set of intervals in $\I$ that do not contain smaller intervals in $\I$, so $\I'$ is the set of atoms of $\L$. By Lemma \ref{l-atoms}, if $\cup_{I\in \I'}I\ne [1,n)$ then $\mu(\hat0, \hat1) =0$. So we may assume that $\cup_{I\in \I'}I= [1,n)$. Again by Lemma \ref{l-atoms}, $\mu(\hat0, \hat1)$ depends only on $\I'$, so we may now assume that $\I'=\I$. 
Then there is an alphabet $A$, a reduced set of forbidden words $F\subseteq A^*$, and a word $w$ such that $I_1, I_2,\dots, I_m$ are the intervals of the forbidden subwords of $w$. (For example, we may take $w$ to be a word of length $n$ with distinct letters and take $F$ to be the set of subwords corresponding to the intervals $I_1,\dots, I_m$.)

By Lemma \ref{l-atoms}, $\mu(\hat0, \hat1) =\sum_I (-1)^I$, where the sum is over all subsets $I$ of $\I$ 
with union $[1,n)$. Thus by the definition of the cluster polynomial and Lemma \ref{l-cluster}, we have
$\mu(\hat0, \hat1) = P_{F,w}(-1)$, which by Theorem \ref{t-main} is 0, 1, or $-1$.
\end{proof}

\longthanks{I would like to thank an anonymous referee for bring Dotsenko and Khoroshkin's work \cite{dk} to my attention and I would like to thank Vladimir Dotsenko for telling me about \cite{dgt} and \cite{iv}.}

\bibliographystyle{amsplain-ac}
\bibliography{gjrefs}

\providecommand{\bysame}{\leavevmode\hbox to3em{\hrulefill}\thinspace}
\providecommand{\MR}{\relax\ifhmode\unskip\space\fi MR }
\providecommand{\MRhref}[2]{%
  \href{http://www.ams.org/mathscinet-getitem?mr=#1}{#2}
}
\providecommand{\href}[2]{#2}
\begin{thebibliography}{10}

\bibitem{anick}
David~J. Anick, \emph{On the homology of associative algebras}, Trans. Amer.
  Math. Soc. \textbf{296} (1986), no.~2, 641--659.

\bibitem{bcn}
Fr\'{e}d\'{e}rique Bassino, Julien Cl\'{e}ment, and Pierre Nicod\`eme,
  \emph{Counting occurrences for a finite set of words: combinatorial methods},
  ACM Trans. Algorithms \textbf{8} (2012), no.~3, Art. 31, 28 pp.

\bibitem{csv}
L.~Carlitz, Richard Scoville, and Theresa Vaughan, \emph{Enumeration of pairs
  of sequences by rises, falls and levels}, Manuscripta Math. \textbf{19}
  (1976), no.~3, 211--243.

\bibitem{dgt}
Vladimir Dotsenko, Vincent G\'elinas, and Pedro Tamaroff, \emph{Finite
  generation for {H}ochschild cohomology of {G}orenstein monomial algebras},
  \arxiv{1909.00487}{KT}, 2019.

\bibitem{dk}
Vladimir Dotsenko and Anton Khoroshkin, \emph{Shuffle algebras, homology, and
  consecutive pattern avoidance}, Algebra Number Theory \textbf{7} (2013),
  no.~3, 673--700.

\bibitem{froberg}
Ralph Fr{\"o}berg, \emph{Determination of a class of {P}oincar\'e series},
  Math. Scand. \textbf{37} (1975), no.~1, 29--39.

\bibitem{gthesis}
Ira~Martin Gessel, \emph{Generating functions and enumeration of sequences},
  Ph.D.~Thesis, Massachusetts Institute of Technology, 1977.

\bibitem{gj}
I.~P. Goulden and D.~M. Jackson, \emph{An inversion theorem for cluster
  decompositions of sequences with distinguished subsequences}, J. London Math.
  Soc. (2) \textbf{20} (1979), no.~3, 567--576.

\bibitem{gjbook}
\bysame, \emph{Combinatorial enumeration}, John Wiley \& Sons, Inc., New York,
  1983.

\bibitem{greene}
Curtis Greene, \emph{A class of lattices with {M}{\"o}bius function {$\pm
  1,0$}}, European J. Combin. \textbf{9} (1988), no.~3, 225--240.

\bibitem{iv}
Natalie Iyudu and Ioannis Vlassopoulos, \emph{Homologies of monomial operads
  and algebras}, \arxiv{2008.00985}{RA}, 2020.

\bibitem{nz}
John Noonan and Doron Zeilberger, \emph{The {G}oulden-{J}ackson cluster method:
  extensions, applications and implementations}, J. Differ. Equations Appl.
  \textbf{5} (1999), no.~4--5, 355--377.

\bibitem{rota}
Gian-Carlo Rota, \emph{On the foundations of combinatorial theory. {I}.
  {T}heory of {M}{\"o}bius functions}, Z. Wahrscheinlichkeitstheorie und Verw.
  Gebiete \textbf{2} (1964), 340--368.

\bibitem{wang}
Chao-Jen Wang, \emph{Applications of the {G}oulden-{J}ackson cluster method to
  counting {D}yck paths by occurrences of subwords}, Ph.D. thesis, Brandeis
  University, 2011.

\bibitem{zhuang}
Yan Zhuang, \emph{A generalized {G}oulden-{J}ackson cluster method and lattice
  path enumeration}, Discrete Math. \textbf{341} (2018), no.~2, 358--379.

\bibitem{zhuang-lift}
\bysame, \emph{A lifting of the {G}oulden--{J}ackson cluster method to the
  {M}alvenuto--{R}eutenauer algebra}, \arxiv{2108.10309}{CO}, 2021.

\end{thebibliography}
\end{document}